\documentclass[12pt,reqno]{amsart}

\usepackage{verbatim}
\usepackage{amssymb}

\newcommand{\F}{{\mathbb F}}

\newcommand{\hA}{{\widehat A}}
\newcommand{\hB}{{\widehat B}}
\newcommand{\hD}{{\widehat D}}
\newcommand{\hG}{{\widehat G}}
\newcommand{\hS}{{\widehat S}}

\newcommand{\alp}{\alpha}
\newcommand{\gam}{\gamma}
\newcommand{\lam}{\lambda}
\newcommand{\sig}{\sigma}

\newcommand{\E}{{\mathsf E}}

\newcommand{\longc}{,\dotsc,}

\newcommand{\longe}{=\dotsb=}

\newcommand{\sbs}{\subset}
\newcommand{\seq}{\subseteq}
\newcommand{\stm}{\setminus}

\newtheorem{lemma}{Lemma}
\newtheorem{theorem}{Theorem}
\newtheorem{corollary}{Corollary}

\newcommand{\refc}[1]{\ref{c:#1}}
\newcommand{\refl}[1]{\ref{l:#1}}
\newcommand{\reft}[1]{\ref{t:#1}}
\newcommand{\refs}[1]{\ref{s:#1}}
\newcommand{\refb}[1]{\cite{b:#1}}
\newcommand{\refe}[1]{\eqref{e:#1}}

\title[Small doubling in prime-order groups]%
  {Small doubling in prime-order groups: \\ from $2.4$ to $2.6$}

\author{Vsevolod F. Lev}
\email{seva@math.haifa.ac.il}
\address[Vsevolod Lev]{Department of mathematics,
  the University of Haifa at Oranim, Tivon 36006, Israel}

\author{Ilya D. Shkredov}
\email{ilya.shkredov@gmail.com}
\address[Ilya Shkredov]{Steklov Mathematical Institute, ul. Gubkina, 8,
    Moscow, Russia, 119991, and
	IITP RAS, Bolshoy Karetny per. 19, Moscow, Russia, 127994, and
	MIPT, Institutskii per. 9, Dolgoprudnii, Russia, 141701}

\begin{document}
\baselineskip=16pt

\begin{abstract}
Improving upon the results of Freiman and Candela-Serra-Spiegel, we show that
for a non-empty subset $A\seq\F_p$ with $p$ prime and $|A|<0.0045p$, (i) if
$|A+A|<2.59|A|-3$ and $|A|>100$, then $A$ is contained in an arithmetic
progression of size $|A+A|-|A|+1$, and (ii) if $|A-A|<2.6|A|-3$, then $A$ is
contained in an arithmetic progression of size $|A-A|-|A|+1$.

The improvement comes from using the properties of higher energies.
\end{abstract}

\maketitle

\section{Introduction. Summary of Results}

The \emph{sumset} and the \emph{difference set} of the subsets $A$ and $B$ of
an additively written abelian group are defined by
\begin{align*}
  A+B &= \{a+b\colon a\in A,\,b\in B \}
  \intertext{and \vskip -0.1in}
  A-B &= \{a+b\colon a\in A,\,b\in B \},
\end{align*}
respectively. We are mostly concerned with the groups of prime order which
are identified with the additive group of the corresponding field and,
accordingly, denoted $\F_p$; here $p$ is the order of the group.

The Cauchy-Davenport theorem asserts that if neither of $A,B\seq\F_p$ is
empty, then
  $$ |A+B| \ge \min\{|A|+|B|-1,p\}. $$
This basic theorem, proved by Cauchy~\refb{c} and independently rediscovered
by Davenport~\cite{b:d1,b:d2}, is arguably the earliest result in the area of
additive combinatorics.

The case of equality in the Cauchy-Davenport theorem was investigated by
Vosper.
\begin{theorem}[Vosper~\cite{b:v1,b:v2}]\label{t:v}
Let $p$ be a prime. If $A,B\seq\F_p$ satisfy $|A|,|B|\ge 2$ and $|A+B|\le
p-2$, then $|A+B|\ge|A|+|B|$ unless $A$ and $B$ are arithmetic progressions
sharing the same common difference.
\end{theorem}

A far-reaching extension of Vosper's theorem, due to Freiman, establishes the
structure of sets $A\seq\F_p$ with the doubling coefficient $|A+A|/|A|$ up to
$2.4$.
\begin{theorem}[Freiman~\refb{f}]\label{t:f}
Let $p$ be a prime. If $A\seq\F_p$ satisfies $|A+A|<2.4|A|-3$ and $|A|<p/35$,
then $A$ is contained in an arithmetic progression with at most $|A+A|-|A|+1$
terms.
\end{theorem}

Theorem~\reft{f} is commonly referred to as \emph{Freiman's
$2.4$-theorem}.

While the expression $|A+A|-|A|+1$ in Theorem~\reft{f} is sharp, the
assumptions $|A+A|<2.4|A|-3$ and $|A|<p/35$ are certainly not and,
conjecturally, can be substantially relaxed. Indeed, some improvements along
these lines have been obtained. For instance, as it follows from a general
result by Green and Ruzsa~\refb{gr}, the conclusion of Theorem~\reft{f} holds
true provided that $|A+A|<3|A|-3$ (which is the best possible bound), and
that $A$ is very small as compared to $p$: namely, $|A|<96^{-108}p$. Two more
results to mention are due to Rodseth~\refb{r} (relaxing the density
assumption in Theorem~\reft{f} to $|A|<p/10.7$), and
Candela-Serra-Spiegel~\refb{css} (replacing the assumptions with
$|A+A|<2.48|A|-7$ and $|A|<10^{-10}p$).

We recommend the interested reader to check~\refb{css} for more discussion
and historical comments.

In this paper we make yet another step in the indicated direction,
improving the constants further and establishing a similar result for the
difference set $A-A$.

\begin{theorem}\label{t:diff}
Let $p$ be a prime, and suppose that $A\seq\F_p$ satisfies $|A|<0.0045p$. If
$|A-A|<2.6|A|-3$, then $A$ is contained in an arithmetic progression with at
most $|A-A|-|A|+1$ terms.
\end{theorem}

\begin{theorem}\label{t:sum}
Let $p$ be a prime, and suppose that $A\seq\F_p$ satisfies $100<|A|<0.0045p$.
If $|A+A|<2.59|A|-3$, then $A$ is contained in an arithmetic progression with
at most $|A+A|-|A|+1$ terms.
\end{theorem}

Our method allows for further slight improvements, but we tried to keep a
reasonable balance to obtain good constants while avoiding excessively
technical computations.

The proofs of Theorems~\reft{diff} and~\reft{sum} presented in
Section~\refs{proofs} follow, from some point on, the familiar path involving
Fourier bias and partial rectification. The major novelty is that we use an
argument of combinatorial nature, based on the properties of higher energies,
to obtain a bias larger than that given by the standard reasoning.

In the appendix we apply our approach to obtain large Fourier bias for the
indicator function of a small-difference set in the general settings of an
arbitrary finite abelian group.

\section{Notation and the Toolbox}

In this section we gather the notation and results used in
Section~\refs{proofs} to prove Theorems~\reft{diff} and~\reft{sum}.

We will occasionally identify sets with their indicator functions; thus, for
instance, for a subset $A$ of a finite abelian group $G$, we have $\sum_{x\in
G}A(x)=|A|$. The non-normalized Fourier coefficients of $A$ are denoted
$\hA$; that is,
  $$ \hA(\chi)=\sum_{a\in A}\chi(a),\quad \chi\in\hG. $$
Hence, $\hA(1)=|A|$ (where $1$ denotes the principal character), and the
Parseval identity reads
  $$ \sum_{\chi\in\hG} |\hA(\chi)|^2 = |A||G|. $$

For a finite subset $A$ and an element $x$ of an abelian group, we let
$A_x:=A\cap(A+x)$; therefore, $|A_x|$ is the number of representations of
$x$ as a difference of two elements of $A$, and in particular $|A_x|=0$ if
$x\notin A-A$. We have
  $$ \sum_{x\in A-A} |A_x| = |A|^2 $$
and
  $$ A_x-A \seq (A-A)_x. $$
The later relation, often called the \emph{Katz-Koester observation}
\refb{kk}, can be proved as follows:
\begin{multline*}
  A_x-A = (A\cap(A+x))-A \seq (A-A)\cap ((A+x)-A) \\
                                       = (A-A) \cap ((A-A)+x) = (A-A)_x.
\end{multline*}
The sum version of the Katz-Koester observation is
  $$ A_x+A \seq (A+A)_x. $$

The \emph{common energy} $\E(A,B)$ of finite subsets $A$ and $B$ of an
abelian group $G$ is the number of quadruples
 $(a_1,a_2,b_1,b_2)\in A^2\times B^2$ such that $a_1-a_2=b_1-b_2$;
equivalently,
  $$ \E(A,B) = \sum_{x\in G} |A_x||B_x|. $$
Also, if $G$ is finite, then
  $$ \E(A,B) = \frac1{|G|}\,\sum_{\chi\in\hG} |\hA(\chi)|^2 |\hB(\chi)|^2. $$
We write $\E(A)$ as a commonly used abbreviation of $\E(A,A)$. For $k>0$ we
set
  $$ \E_k(A) := \sum_{x\in A-A} |A_x|^k; $$
thus, $\E_2(A)=\E(A)$, and if $k$ is an integer, then
  $$ \E_k (A) = |\{(a_1\longc a_k,b_1\longc b_k)\in A^{2k} \colon
                                              a_1-b_1 \longe a_k-b_k \}|. $$

For real $u\le v$, by $[u,v]$ we denote the set of all integers $u\le n\le
v$, and also the ``canonical'' image of this set in $\F_p$.


The following theorem follows easily from the results of~\refb{f1}.
\begin{theorem}[Freiman~\refb{f1}]\label{t:F3n-3}
Suppose that $A$ is a finite set of integers. If $|A+A|\le 3|A|-4$, then $A$
is contained in an arithmetic progression with at most $|A+A|-|A|+1$ terms.
Similarly, if $|A-A|\le 3|A|-4$, then $A$ is contained in an arithmetic
progression with at most $|A-A|-|A|+1$ terms.
\end{theorem}

We need two more lemmas due to Freiman; the former originates from~\refb{f2},
while the latter is implicit in~\refb{f} and in fact in any exposition of the
proof of Theorem~\reft{f}, such as~\cite[Section~2.8]{b:n}.
\begin{lemma}[Freiman~\refb{f2}]\label{l:Fpoints}
Suppose that $Z$ is a finite subset of the unit circle on the complex
plane. If
  $$ \Big| \sum_{z\in Z} z \Big| = \eta |Z|, $$
then there is an open arc of the circle of the angle measure $\pi$ containing
at least $\frac12(1+\eta)|Z|$ elements of $Z$.
\end{lemma}

\begin{lemma}[Freiman~\refb{f}]\label{l:fFB}
Suppose that $p$ is a prime, and that a subset $A\seq\F_p$ satisfies
$|A|<p/12$ and $|A+A|<K|A|-3$ with some $2\le K\le3$. If there is an
arithmetic progression in $\F_p$ with $(p+1)/2$ terms, containing at least
$\frac13\,K|A|$ elements of $A$, then, indeed, the whole set $A$ is contained
in an arithmetic progression with at most $|A+A|-|A|+1$ terms.
\end{lemma}

An essentially identical statement holds true for the subsets $A\seq\F_p$
with the difference set satisfying $|A-A|<K|A|-3\ (2\le K\le 3)$. For
self-completeness, we provide a very brief sketch of the proof, addressing
both the sum and the difference sets together.

\begin{proof}[Proof of Lemma~\refl{fFB}]
Scaling and translating $A$ appropriately, we assume without loss of
generality that, with $A':=A\cap[0,(p-1)/2]$, and with $A''$ defined to be
the inverse image of $A'$ in $[0,(p-1)/2]$ under the canonical homomorphism,
we have $0\in A''$, $\gcd(A'')=1$, and $|A''|=|A'|\ge\frac13\,K|A|$. Thus,
  $$ |A''\pm A''| = |A'\pm A'| \le |A\pm A| < K|A| - 3 \le 3|A''| - 3, $$
and by Theorem~\reft{F3n-3}, the set $A''$ is contained in an arithmetic
progression with at most $|A''\pm A''|-|A''|+1$ terms; that is, letting
$l:=\max A''$,
  $$ l \le |A''\pm A''|-|A''| \le |A\pm A| - \frac13\,K|A|
                                              < \frac23\,K|A|-3 < p/6. $$
Therefore, $A'\seq[0,l]$ with $l<p/6$, and it follows that
$A'+A'-A'\seq[-l,2l]$, showing that for any group element $x\in[2l+1,p-l-1]$,
the sets $x+A'$ and $A'+A'$ are disjoint. If we had $x\in A$, then, in view
of the Cauchy-Davenport theorem, we would have
  $$ |A\pm A| \ge |A'\pm A'| + |x+A'| \ge 3|A'|-1 \ge K|A| - 1, $$
a contradiction. Thus, $A\seq[-l,2l]$, showing that $A$ is contained in an
arithmetic progression with at most $(p+1)/2$ terms. Considering now the
inverse image of $A$ in the interval $[-l,2l]$, we conclude, as above, that
this image, and therefore the set $A$ itself, are in fact contained in
arithmetic progressions with at most $|A\pm A|-|A|+1$ terms, as wanted.
\end{proof}

Combining Lemmas~\refl{Fpoints} and~\refl{fFB} we obtain
\begin{corollary}\label{c:f}
Let $p$ be a prime, and suppose that $A\seq\F_p$ is a set such that
$|A|<p/12$ and $|A\pm A|<K|A|-3$ with some $2\le K\le 3$. If there exists a
nonprincipal character $\chi\in\widehat{\F_p}$ such that
$|\hA(\chi)|\ge\eta|A|$, where $\eta\in[0,1]$ satisfies
$\frac12(1+\eta)\ge\frac13K$, then $A$ is contained in an arithmetic
progression with at most $|A\pm A|-|A|+1$ terms.
\end{corollary}

Finally, we state and prove a lemma which bounds the number of \emph{Schur
triples} contained in a subset of $\F_p$.
\begin{lemma}\label{l:3/4}
Let $p$ be a prime. For any set $D\sbs\F_p$ with $|D|$ odd and
$|D|\le(2p+1)/3$, we have 	
  $$ \sum_{x,y\in D} D(x-y) \le \frac34\,|D|^2+\frac14. $$
\end{lemma}

\begin{proof}
Let $n:=(|D|-1)/2$. The sum in the left-hand side counts triples $(x,y,z)\in
D^3$ with $x-y-z=0$. By \cite[Theorem~1]{b:l}, the number of such triples can
only increase if $D$ is replaced with the interval $[-n,n]\seq\F_p$.
Therefore, the sum in question does not exceed
\begin{align*}
  |\{ (x,y,z)\in [-n,n]^3 &\colon x-y-z=0 \}| \\
    &= |\{ (x,y)\in[-n,n]^2\colon y-x\in[-n,n]\}| \\
    &= \sum_{x\in[-n,n]} |[x-n,x+n]\cap[-n,n]| \\
    &= (2n+1) + 2\sum_{x=1}^n |[x-n,n]| \\
    &= (2n+1) + 2 \sum_{x=1}^n (2n+1-x) \\
    &= \frac34\,(2n+1)^2+\frac14;
\end{align*}
here all intervals are subsets of $\F_p$, and the assumption
$2n+1=|D|\le(2p+1)/3$ ensures that $[x-n,x+n]\cap[-n,n]=[x-n,n]$ whenever
$x\in[1,n]$.
\end{proof}

\section{Proofs of Theorems~\reft{diff} and \reft{sum}}\label{s:proofs}

For a subset $A\seq\F_p$ with $|A-A|=K|A|$, as an immediate application of
the Cauchy-Schwarz inequality we have $\E(A)\ge K^{-1}|A|^3$. We start with a
lemma improving this trivial bound; the lemma will be used in the proof of
Theorem~\reft{diff}.

\begin{lemma}\label{l:CS}
Let $p$ be a prime, and suppose that a subset $A\subset \F_p$ satisfies
$|A-A|=K|A|<p/2$. Then
  $$ \E(A) \ge
         \left( \frac1K + \frac1{3K(K+2)}\,(1-|A|^{-2}) \right) |A|^3. $$
\end{lemma}

\begin{proof}
Write $D:=A-A$ and $\lam:=|A|^2/|D|$, and let
  $$ F(x) := |A_x| - \lam D(x),\quad x\in\F_p $$
and
  $$ \sig_k := \sum_{x\in\F_p} F^k(x), $$
where $k$ is a positive integer. We have
\begin{gather}
  \sig_1  = \sum_{x\in\F_p} F(x)=0, \notag \\
  \sig_2  = \sum_{x\in\F_p} F^2 (x) =  \E(A) - 2\lam |A|^2 + \lam^2|D|
                                 =  \E(A) - \frac{|A|^4}{|D|}, \label{e:sig2}
\end{gather}
and
\begin{align}
  \sig_3  &=   \sum_{x\in\F_p} F^3 (x) \notag \\
          &=  \E_3(A) - 3\lam\E(A) + 3\lam^2|A|^2 - \lam^3 |D| \notag \\
          &=  \E_3(A) - 3\lam \left(\E(A) - \frac{|A|^4}{|D|}\right)
                                         - \frac{|A|^6}{|D|^2} \notag \\
          &=  \E_3(A) - 3\frac{|A|^2}{|D|}\sig_2
                                      - \frac{|A|^6}{|D|^2}. \label{e:sig3}
\end{align}

Also, from $F(x)\le |A|-\lam = |A|-|A|^2/|D|=(1-|A|/|D|)|A|$ we get
\begin{equation}\label{e:sig23}
  \sig_3 \le \Big( 1- \frac{|A|}{|D|} \Big) |A| \sig_2.
\end{equation}
From~\refe{sig2}--\refe{sig23},
\begin{multline}\label{e:E3Upper}
  \E_3(A) = \sig_3 + 3\,\frac{|A|^2}{|D|}\,\sig_2  + \frac{|A|^6}{|D|^2}
      \le \Big( 1 + 2\,\frac{|A|}{|D|} \Big) |A| \sig_2
                                                + \frac{|A|^6}{|D|^2} \\
      =\Big( 1 + 2\,\frac{|A|}{|D|} \Big)
          \Big( \E(A)-\frac{|A|^4}{|D|}\Big) |A| + \frac{|A|^6}{|D|^2}.
\end{multline}

We now use the basic properties of higher energies from \refb{ss} to estimate
$\E_3(A)$ from below. To this end, we observe that
\begin{align*}
  \sum_{x,y\in\F_p} |A\cap (A+x) \cap (A+y)|
     &= \sum_{x,y\in\F_p} |\{a\in A\colon a-x,a-y\in A\}| \\
     &= \sum_{a\in A} |\{(x,y)\in\F_p^2\colon a-x,a-y\in A \}| \\
     &= \sum_{a\in A} |A|^2 \\
     &= |A|^3.
\end{align*}
In a similar way, considering pairs $(a,b)\in A^2$ with $a-x,a-y,b-x,b-y\in
A$, we get
  $$ \sum_{x,y\in\F_p} |A\cap (A+x) \cap (A+y)|^2 = \E_3(A). $$
Furthermore, the number of non-zero summands in these sums is the number of
pairs $(x,y)$ such that there exist $a,b,c\in A$ with $x=a-b$ and $y=a-c$;
that is, the number of pairs representable in the form $(a-b,a-c)$, where
$a,b,c\in A$. Consequently, using the Cauchy--Schwarz inequality we obtain
\begin{align*}
 |A|^6 &= \Big( \sum_{x,y\in\F_p} |A\cap (A+x) \cap (A+y)| \Big)^2 \\
       &\le \sum_{x,y\in\F_p} |A\cap (A+x) \cap (A+y)|^2 \cdot
                                     |\{ (b-a,c-a)\colon a,b,c \in A \}| \\
	   &\le \E_3 (A) \sum_{x,y\in D} D(x-y).
\end{align*}
Applying Lemma~\refl{3/4}, we conclude that
  $$ \E_3(A) \cdot \Big(\frac34\,|D|^2+\frac14\Big) \ge |A|^6. $$
Since $\E_3(A)\le|A|^4$, this leads to
  $$ \E_3(A) \ge \frac43\frac{|A|^6}{|D|^2} - \frac13\frac{|A|^4}{|D|^2}. $$
From this inequality and~\refe{E3Upper},
  $$ \frac43 \frac{|A|^6}{|D|^2}-\frac13\,\frac{|A|^4}{|D|^2}
      \le  \Big( 1 + 2\,\frac{|A|}{|D|} \Big)
         \Big( \E(A)-\frac{|A|^4}{|D|} \Big) |A| + \frac{|A|^6}{|D|^2}, $$
and a short computation gives
\begin{align*}
  \E(A) &\ge \frac{|A|^4}{|D|} + \frac13\,
      \frac{|A|^5-|A|^3}{(|D|+2|A)|D|} \\
        &= \Big( \frac1K + \frac{1}{3K(K+2)}
                          - \frac{1}{3K(K+2)|A|^2}\Big) |A|^3.
\end{align*}
\end{proof}

We are now ready to prove Theorems~\reft{diff} and~\reft{sum}.
\begin{proof}[Proof of Theorem~\reft{diff}]
Using the Cauchy-Davenport and Vosper theorems, it is easy to verify the
assertion for $|A|\le 4$; we therefore assume throughout that $|A|\ge 5$.
We set $D:=A-A$ and $K:=|D|/|A|$; thus, $K<2.6$.

Let $\eta$ be defined by
$\max\{|\hA(\chi)|\colon\chi\in\widehat{\F_p}\stm\{1\}\}=\eta|A|$, and
let $\alp:=|A|/p$ be the density of $A$. In view of
\begin{equation}\label{e:Erho}
  \E(A) = p^{-1} \sum_{\chi\in\widehat{\F_p}} |\hA(\chi)|^4
                                                \le \alp|A|^3 + \eta^2 |A|^3,
\end{equation}
from Lemma~\refl{CS} we obtain
\begin{equation}\label{e:weak}
  \eta^2 \ge \frac1K + \frac{1}{3K(K+2)} - \frac{1}{3K(K+2)|A|^2} - \alp.
\end{equation}

With some extra effort, we now prove a slightly better bound, in the spirit
of \cite{b:ss} where a short proof of a Katz-Koester energy
result~\cite{b:kk} is presented.

Consider the sum
\begin{equation}\label{e:difsum}
   \sum_{x\in D} |A_x| |A-A_x|.
\end{equation}
The term corresponding to $x=0$ is $|A||D|$, while for every element $x\in
D\stm\{0\}$ we have $|A-A_x|\ge |A|+|A_x|-1$ by the Cauchy-Davenport theorem.
Therefore
\begin{align}
  \sum_{x\in D} |A_x| |A-A_x|
    &\ge |A||D| + \sum_{x\in D\stm\{0\}} |A_x| (|A| + |A_x| - 1) \notag \\
    &=   |A||D| + (|A|^2-|A|)|A| + (\E(A) - |A|^2) - (|D|-1) \notag \\
    &> |A|^3 + \E(A) + (K-2)|A|^2 - K|A|. \label{e:0512}
\end{align}

Combining this estimate with the estimate of Lemma~\refl{CS} and the
Katz-Koester observation $|A-A_x|\le |D_x|$, we get
\begin{align*}
  p^{-1}|A|^2|D|^2 + & p^{-1} \eta^2|A|^2 (p-|D|)|D| \\
     &\ge p^{-1}\, \sum_{\chi\in\widehat{\F_p}} |\hA(\chi)|^2 |\hD(\chi)|^2 \\
     &=   \sum_{x\in\F_p} |A_x| |D_x| \\
     &= \sum_{x\in D} |A_x| |D_x| \\
     &> |A|^3 + \E(A) + (K-2)|A|^2  - K|A| \\
	 &\ge  \left(1+ \frac{1}{K} + \frac{1}{3K(K+2)}
                     - \frac1{3K(K+2)|A|^2}
                         + \frac{K-2}{|A|} - \frac{K}{|A|^2} \right)\,|A|^3.
\end{align*}
Asymptotically, in the regime where $|A|$ grows, but $\alp K^2=o(1)$, this
yields
  $$ \eta^2 \ge \frac1K+\frac1{K^2}+\frac1{3K^2(K+2)} + o(1) $$
(which is worth comparison against~\refe{weak}).

To obtain an explicit version of this estimate suitable for our present
purposes, we let $\eta_0:=\frac23\cdot 2.6-1$ and notice that if
$\eta<\eta_0$, then the last computation gives
  $$ \alp K^2 + K(1-\alp K)\eta_0^2
       \ge 1 + \frac{1}{K} + \frac{1}{3K(K+2)} - \frac1{3K(K+2)|A|^2}
                        + \frac{K-2}{|A|} - \frac{K}{|A|^2}; $$
equivalently,
  $$ (1-\eta_0^2) \alp K^2 + \Big(\eta_0^2-\frac1{|A|}\Big)K+\frac{2}{|A|}
       \ge 1 + \frac{1}{K} + \frac{1}{3K(K+2)}
                     - \frac1{3K(K+2)|A|^2} - \frac{K}{|A|^2}. $$
Since the left-hand side is an increasing function of $K$, while the
right-hand side is decreasing, the inequality remains valid with $K$
substituted by $2.6$; making the substitution, dividing through by
$(1-\eta_0^2)K^2$, and computing numerically, we obtain
  $$ \alp > 0.0045 + \frac{0.1920}{|A|} - \frac{0.8410}{|A|^2}
       > 0.0045, $$
contrary to the assumptions. Thus, $\eta\ge\eta_0$, and an application of
Corollary~\refc{f} completes the proof.
\end{proof}

The proof of Theorem~\reft{sum} is in fact a simplified version of that of
Theorem~\reft{diff}, due to the fact that some components of the proof
specific for the differences cannot be reproduced for the sums, and are thus
omitted. As a result, the argument is somewhat shorter, but the eventual
estimate is slightly less precise.
\begin{proof}[Proof of Theorem~\reft{sum}]
As in the proof of Theorem~\reft{diff}, we write $D:=A-A$ and
$\alp:=|A|/p$, and define $\eta$ by
$\max\{|\hA(\chi)|\colon\chi\in\widehat{\F_p}\stm\{1\}\}=\eta|A|$. We
also let $S:=A+A$ and $K:=|S|/|A|$.

Instead of Lemma~\refl{CS}, our starting point is the estimate
$\E(A)\ge|A|^4/|S|$ following readily from the Cauchy-Schwarz inequality.
Instead of~\refe{difsum}, we now consider the sum $\sum_{x\in D}
|A_x||A+A_x|$ for which, applying the Cauchy-Davenport theorem, we get
\begin{align*}
  \sum_{x\in D} |A_x||A+A_x|
     &\ge \sum_{x\in D} |A_x|(|A|+|A_x|-1) - |A|(2|A|-1) + |A||S| \\
     &=   |A|^3 + \E(A) + |A||S|-3|A|^2 + |A|,
\end{align*}
cf.~\refe{0512}. Using, on the other hand, the estimate
 $|A+A_x|\le |S_{x}|$, we obtain
\begin{align*}
  p^{-1}|A|^2|S|^2 + & p^{-1} \eta^2|A|^2 (p-|S|)|S| \\
     &\ge p^{-1}\, \sum_{\chi\in\widehat{\F_p}} |\hA(\chi)|^2 |\hS(\chi)|^2 \\
     &=   \sum_{x\in D} |A_x| |S_x| \\
     &\ge \sum_{x\in D} |A_x||A+A_x| \\
     &\ge |A|^3 + \E(A) + |A||S|-3|A|^2 + |A| \\
	 &\ge \left(1+ \frac{1}{K} - \frac{3-K}{|A|} \right)\,|A|^3.
\end{align*}
As a result,
  $$ \eta^2 \ge \frac1{K(1-\alp K)}\,
                \left(1+ \frac{1}{K} - \frac{3-K}{|A|} - \alp K^2  \right). $$
Let $\eta_0:=\frac23\cdot2.59-1$. If we had $\eta<\eta_0$, this would
imply
  $$ K(1-\alp K) \eta_0^2 > 1 + \frac{1}{K} - \frac{3-K}{|A|} - \alp K^2; $$
that is,
  $$ (1-\eta_0^2)\alp K^2 + \Big(\eta_0^2-\frac1{|A|}\Big)K
                                             + \frac{3}{|A|} > 1 + \frac1K. $$
Since the left-hand side is an increasing function of $K$, while the
right-hand side is decreasing, in view of $K<2.59$ we would conclude that the
last inequality stays true if $K$ gets substituted by $2.59$; substituting,
normalizing, and computing numerically, we obtain
  $$ \alp + \frac{0.1296}{|A|}> 0.0058, $$
contradicting the assumptions $\alp<0.0045$ and $|A|>100$.

Thus, $\eta\ge\eta_0$, and we invoke Corollary~\refc{f} to complete the
proof.
\end{proof}

\appendix
\section{Arbitrary groups}

The standard argument shows that for a subset $A$ of an arbitrary finite
abelian group $G$, keeping the notation $K$ for the doubling coefficient
$|A+A|/|A|$, and $\eta|A|$ for the largest absolute value of a non-trivial
Fourier coefficient of the indicator function of $A$, one has
  $$ \eta \ge \frac1{\sqrt K}\,\sqrt{\frac{1-\gam}{1-\alp}}, $$
where $\alp:=|A|/|G|$ and $\gam:=|A+A|/|G|$ are the densities of $A$ and
$A+A$, respectively.

The same estimate holds true for the difference set $A-A$. In any case,
assuming $\gam=o(1)$, we get
\begin{equation}\label{f:rho_stupid}
	\eta \ge \frac{1+o(1)}{\sqrt{K}}.
\end{equation}
In the case of difference sets we have the following slight improvement which
basically replaces the term $o(1)$ in \eqref{f:rho_stupid} with a positive
constant.

\begin{theorem}\label{t:G_rho}
Let $A$ be a non-empty subset of a finite abelian group $G$ of density
$\alp=|A|/|G|$. If $|A-A|=K|A|$, then
  $$ \eta
     \ge \left(\frac{1+\sqrt{5}}{2} \right)^{1/2}
                                      \frac{1+O(K^3\alp)}{\sqrt{K+1}}, $$
where $\eta$ is defined by
$\max\{|\hA(\chi)|\colon\chi\in\hG\stm\{1\}\}=\eta|A|$.
\end{theorem}

\begin{proof}
From~\refe{E3Upper} and~\refe{Erho} (which are valid in any finite abelian
group, not necessarily of prime order),
\begin{equation}\label{e:rho_E_3}
  \E_3(A) \le \Big(1+\frac2K\Big)
          \Big(\eta^2-\frac{1}{K}+\alp\Big)|A|^4 + \frac{|A|^4}{K^2}.
\end{equation}
On the other hand, letting $D:=A-A$, by H\"older's inequality,
  $$ |A|^2 = \sum_{x\in D} |A_x|
                        \le |D|^{1/3} \Big( \sum_x |A_x|^{3/2} \Big)^{2/3}
                              = |D|^{1/3} (\E_{3/2}(A))^{2/3}; $$
that is,
  $$ \E_{3/2}(A) \ge K^{-1/2}|A|^{5/2}. $$
Combining this with the estimate
\begin{equation*}
  |A|^2\, \E^2_{3/2}(A) \le \E_3(A) \E(A,D)
\end{equation*}
established in \cite[Corollary~4.3]{b:s_mixed}, and then
with~\refe{rho_E_3} and
  $$ \E(A,D) = p^{-1}\sum_\chi |\hA(\chi)|^2 |\hD(\chi)|^2 \le
                              \alp K^2|A|^3 + \eta^2 K(1-\alp K)|A|^3, $$
we obtain
\begin{multline*}
   K^{-1}|A|^7 \le \E_3(A) \E(A,D)
      \le \left( \Big(1+\frac2K\Big)
           \Big(\eta^2-\frac{1}{K}+\alp\Big)|A|^4
                     + \frac{|A|^4}{K^2} \right) \\
      \cdot \big( \alp K^2|A|^3 + \eta^2 K(1-\alp K)|A|^3 \big).
\end{multline*}
This gives
  $$ 1 \le \left( (K+2)(\eta^2K-1)+1 \right)\cdot\eta^2 + O(K^3\alp), $$
and after rearranging the terms,
  $$ \eta^4 K(K+2) - (K+1)\eta^2 - (1+O(K^3\alp)) \ge 0. $$
It follows that
  $$ \Big( (K+1)\eta^2 - \frac12 \Big)^2 \ge \frac54 + O(K^3\alp)
        =\frac54\,(1+O(K^3\alp)), $$
whence
  $$ (K+1)\eta^2 \ge \frac12 + \sqrt{\frac54\,(1+O(K^3\alp))}
            = \frac{1+\sqrt 5}{2}\,(1+O(K^3\alp)),  $$
resulting in
  $$ \eta \ge \left(\frac{1+\sqrt 5}{2}\right)^{1/2}
                                \frac1{\sqrt{K+1}} \, (1+O(K^3\alp)). $$
\end{proof}

If something is known about subgroups of the group $G$ (as, for instance, in
the case where $G=\F_p$), then Fournier--type results~\refb{fou} can be
applied (see \cite[Lemma 7.2]{b:s} for a modern exposition), allowing one to
estimate $\E_{3}(A)$ from below nontrivially, and hence improving
Theorem~\ref{t:G_rho} in this situation.


\smallskip
\end{document}